\documentclass[a4paper]{amsart}
\usepackage{amscd}
\usepackage{amssymb}
\usepackage{algorithm}
\usepackage{algorithmic}

\usepackage[T1]{fontenc}
\usepackage[latin1]{inputenc}
\usepackage{hyperref}
\usepackage[arrow,curve,matrix]{xy}
\usepackage{times}
\usepackage{graphicx}

\DeclareFontFamily{OMS}{rsfs}{\skewchar\font'60}
\DeclareFontShape{OMS}{rsfs}{m}{n}{<-5>rsfs5 <5-7>rsfs7 <7->rsfs10 }{}
\DeclareSymbolFont{rsfs}{OMS}{rsfs}{m}{n}
\DeclareSymbolFontAlphabet{\scr}{rsfs}

\renewcommand{\O}{\sO}%{\mathcal{O}}
\newcommand{\Q}{\mathbb{Q}}

\newcommand\resto[1]{\hbox{\hbox{$\vert{}_{_{#1}}$}}}

\newcommand{\ratmap}{\dasharrow} 
\newcommand{\sA}{\scr{A}}

\newcommand{\sO}{\scr{O}}

\DeclareMathOperator{\Sym}{Sym}

\DeclareMathOperator{\Var}{Var}

\newcounter{thisthm}

\newcommand{\iref}[1]{(\thesection.\the\value{thisthm}.\the\value{#1})}

\theoremstyle{plain}    
\newtheorem{thm}{Theorem}[section]

\numberwithin{equation}{thm}
\numberwithin{figure}{section}
\theoremstyle{plain}    
\newtheorem{cor}[thm]{Corollary}

\newtheorem{conjecture}[thm]{Conjecture}

\theoremstyle{plain}    

\newtheorem{proclaim-special}[thm]{\specialthmname}

\theoremstyle{remark}
\newtheorem{rem}[thm]{Remark}

\newtheorem{claim}[thm]{Claim} %%Delete [...] to re-start numbering
\newtheorem*{claim*}{Claim}

\newtheoremstyle{bozont-remark}{3pt}{3pt}%
     {}%         Body font
     {}%         Indent amount (empty = no indent, \parindent = para indent)
     {\it}% Thm head font
     {.}%        Punctuation after thm head
     {.5em}%     Space after thm head (\newline = linebreak)
     {\thmname{#1}\thmnumber{ #2}: \thmnote{\sc #3}}%         Thm head spec
\theoremstyle{bozont-remark}

\def\factor#1.#2.{\left. \raise 2pt\hbox{$#1$} \right/\hskip -2pt\raise
  -2pt\hbox{$#2$}}

\newlength{\swidth}
\newenvironment{enumerate-p}{
  \begin{enumerate}}
  {\setcounter{equation}{\value{enumi}}\end{enumerate}}

\date{\today}

\author{Stefan Kebekus}
\author{S\'andor J.\ Kov\'acs}

\thanks{Stefan Kebekus was supported in part by the DFG-Forschergruppe
  ``Classification of Algebraic Surfaces and Compact Complex
  Manifolds''.  S\'andor Kov\'acs was supported in part by NSF Grant
  DMS-0554697 and the Craig McKibben and Sarah Merner Endowed
  Professorship in Mathematics.}

\address{Stefan Kebekus, Mathematisches Institut, Universit\"at zu
  K\"oln, Weyertal 86--90, 50931 K\"oln, Germany}
\email{\href{mailto:stefan.kebekus@math.uni-koeln.de}{stefan.kebekus@math.uni-koeln.de}}
\urladdr{\href{http://www.mi.uni-koeln.de/~kebekus}{http://www.mi.uni-koeln.de/$\sim$kebekus}}

\address{\noindent S\'andor Kov\'acs, University of Washington,
  Department of Mathematics, Box 354350, Seattle, WA 98195, U.S.A.}
\email{\href{mailto:kovacs@math.washington.edu}{kovacs@math.washington.edu}}
\urladdr{\href{http://www.math.washington.edu/~kovacs}{http://www.math.washington.edu/$\sim$kovacs}}

\numberwithin{equation}{subsection}

\begin{document}

\title{Families of varieties of general type over compact bases}

\maketitle

\section{Introduction} % and statement of main result}

Let $f: X \to Y$ be a smooth family of canonically polarized complex
varieties over a smooth base. Generalizing the classical Shafarevich
hyperbolicity conjecture, Viehweg conjectured that $Y$ is necessarily
of log general type if the family has maximal variation. We refer to
\cite{KK05} for a precise formulation, for background and for details
about these notions. A somewhat stronger and more precise version of
Viehweg's conjecture was shown in \cite{KK05} in the case where $Y$ is
a quasi-projective surface. Assuming that the minimal model program
holds, this very short paper proves the same result for projective
base manifolds $Y$ of arbitrary dimension.

We recall the two relevant standard conjectures of higher dimensional
algebraic geometry first.

\begin{conjecture}[Minimal Model Program and Abundance for $\kappa =
  0$]\label{conj:1}
  Let $Y$ be a smooth projective variety such that $\kappa(Y)=0$.
  Then there exists a birational map $\lambda : Y \dasharrow
  Y_{\lambda}$ such that the following holds.
  \begin{enumerate}
  \item $Y_\lambda$ is  $\Q$-factorial and has at worst terminal
    singularities.
  \item There exists a number $n$ such that $n K_{Y_\lambda}$ is
    trivial, i.e., $\O_{Y_\lambda}(n K_{Y_\lambda}) = \O_{Y_\lambda}$
  \end{enumerate}
\end{conjecture}

\begin{conjecture}[Abundance for $\kappa = -\infty$]\label{conj:2}
  Let $Y$ be a smooth projective variety. If $\kappa(Y)=-\infty$,
  then $Y$ is uniruled.
\end{conjecture}

\begin{rem}\label{rem:conjinsmaldim}
  Conjectures~\ref{conj:1} and \ref{conj:2} are known to hold for all
  varieties of dimension $\dim Y \leq 3$.
\end{rem}

The main result of this paper is now the following, cf.\
\cite[Conjecture~1.6]{KK05}.

\begin{thm}\label{thm:main}
  Let $Y$ be a smooth projective variety and $f: X \to Y$ a smooth
  family of canonically polarized varieties.  Assume that
  Conjectures~\ref{conj:1} and \ref{conj:2} hold for all varieties $F$
  of dimension $\dim F \leq \dim Y$. Then the following holds.
  \begin{enumerate}
  \item If $\kappa(Y)=-\infty$, then $\Var(f) < \dim Y$.
  \item If $\kappa(Y) \geq 0$, then $\Var(f) \leq \kappa(Y)$.
  \end{enumerate}
\end{thm}
\begin{rem}
  The argumentation of Section~\ref{sec:proof} actually shows a
  slightly stronger result. If $\kappa(Y)=-\infty$, it suffices to
  assume that Conjecture~\ref{conj:2} holds for $Y$. If $\kappa(Y)
  \geq 0$, we need to assume that Conjecture~\ref{conj:1} holds for
  all varieties $F$ of dimension $\dim F = \dim Y - \kappa(Y)$.

  See Theorem~\ref{thm:main-2} below for further generalizations.
\end{rem}

Theorem~\ref{thm:main} and Remark~\ref{rem:conjinsmaldim} immediately
imply the following.

\begin{cor}
  Viehweg's conjecture holds for smooth families of canonically
  polarized varieties over projective base manifolds of dimension
  $\leq 3$. \qed
\end{cor}

\section{Proof of Theorem~\ref*{thm:main}}
\label{sec:proof}

\subsection{The case where $\boldsymbol{\kappa(Y)=-\infty}$} 

The assertion follows immediately from Conjecture~\ref{conj:2} and
from the fact that families of canonically polarized varieties over
rational curves are necessarily isotrivial \cite[Thm.~1]{Kovacs96}.

\subsection{The case where $\boldsymbol{\kappa(Y)=0}$}
\label{subsect:kappa=0}

In this case, we need to show that the family $f$ is isotrivial. We
argue by contradiction and assume that $\Var(f) \geq 1$. By
\cite[Thm.~1.4.i]{VZ02}, this implies that there exists a number $n$
and an invertible subsheaf $\sA \subset \Sym^n\Omega^1_Y$ of
Kodaira-Iitaka dimension $\kappa(\sA) \geq \Var(f) \geq 1$.

By assumption, there exists a birational map $\lambda: Y \ratmap
Y_\lambda$ as discussed in Conjecture~\ref{conj:1}. Resolving the
indeterminacies of $\lambda$ and pulling back the family $f$, we may
assume without loss of generality that $\lambda$ is a morphism, i.e.,
defined everywhere.

Let $C_\lambda \subset Y_\lambda$ be a general complete intersection
curve. Then $C_\lambda$ will avoid the singularities of $Y_\lambda$.
In particular, the restriction $\Omega^1_{Y_\lambda}\resto{C_\lambda}$
is a vector bundle of degree
\begin{equation}\label{eq:1}
  \deg \Omega^1_{Y_\lambda}\resto{C_\lambda} = K_{Y_\lambda}\cdot
  C_\lambda = 0.
\end{equation}

\begin{claim}\label{claim:1}
  The vector bundle $\Omega^1_{Y_\lambda}\resto{C_\lambda}$ is not
  semi-stable.
\end{claim}
\begin{proof}[Proof of Claim~\ref{claim:1}]
  Observe that the curve $C_\lambda$ avoids the fundamental points of
  $\lambda$, and hence that $\lambda$ is an isomorphism in a
  neighborhood of $C_\lambda$. Setting $C := \lambda^{-1}(C_\lambda)$,
  the morphism $\lambda$ induces an isomorphism
  $\Omega^1_{Y_\lambda}\resto{C_\lambda} \cong \Omega^1_Y\resto{C}$.
  This shows that $\Omega^1_{Y_\lambda}\resto{C_\lambda}$ cannot be
  semi-stable, for if it was, its symmetric product $\Sym^n
  \Omega^1_{Y_\lambda}\resto{C_\lambda}$ would also be semistable of
  degree $0$. However, this contradicts the existence of the subsheaf
  $\sA$ whose restriction to $C$ has positive degree.
\end{proof}

To end the proof, observe that \eqref{eq:1} and Claim~\ref{claim:1}
together imply that $\Omega^1_{Y_\lambda}\resto{C_\lambda}$ has an
invertible quotient of negative degree. In this setup, Miyaoka's
uniruledness criterion, cf.\ \cite[Cor.~8.6]{Miy85}, \cite{KST07} or
\cite[Chapt.~2.1]{KS06}, applies to show that $Y$ is uniruled,
contradicting the assumption that $\kappa(Y)=0$.

\subsection{The case where $\boldsymbol{\kappa(Y) > 0}$} 

In this case, consider the Iitaka fibration of $Y$, $\mathfrak i:Y'\to
Z$. Since the Iitaka model is only determined birationally, we may
assume that there exists a birational morphism $Y'\to Y$.  Pulling the
family $f:X\to Y$ back to $Y'$, we may assume that $Y'=Y$, and hence
we may assume that there exists a fibration $\mathfrak i:Y\to Z$ such
that $\dim Z=\kappa(Y)$ and $\kappa(F)=0$ for the general fiber $F$ of
$\mathfrak i$ \cite[Thm.~11.8]{Iitaka82}.  We have seen in
Section~\ref{subsect:kappa=0} that $f\resto F$ is isotrivial and hence
$\Var(f)\leq \dim Y-\dim F =\dim Z =\kappa(Y)$.  This finishes the
proof of Theorem~\ref{thm:main}.

\section{Families of varieties of general type}

Using \cite[Thm.~1.4.iii]{VZ02}, the argumentation of
Section~\ref{sec:proof} immediately gives the following, somewhat
weaker, result for families of varieties of general type.

\begin{thm}\label{thm:main-2}
  Let $Y$ be a smooth projective variety and $f: X \to Y$ a smooth
  family of varieties of general type of maximal variation, i.e.,
  $\Var(f) = \dim Y$. If Conjectures~\ref{conj:1} and \ref{conj:2}
  hold for all varieties $F$ of dimension $\dim F \leq \dim Y$, then
  $Y$ is of general type. \qed
\end{thm}

\providecommand{\bysame}{\leavevmode\hbox to3em{\hrulefill}\thinspace}
\providecommand{\MR}{\relax\ifhmode\unskip\space\fi MR}
\providecommand{\MRhref}[2]{%
  \href{http://www.ams.org/mathscinet-getitem?mr=#1}{#2}
}
\providecommand{\href}[2]{#2}

\end{document}